\documentclass[12pt]{amsart}
\usepackage{amsmath,a4wide}
\usepackage{amssymb,color}
\usepackage[mathcal]{eucal}
\usepackage[all]{xy}
\usepackage{tikz}
\usetikzlibrary{calc,decorations.pathmorphing,shapes}
\usepackage{lineno}
\newtheorem{deff}{Definition}[section]

\newtheorem{theorem}[deff]{Theorem}
\newtheorem{corollary}[deff]{Corollary}

\newtheorem{proposition}[deff]{Proposition}

\newtheorem{em-example}[deff]{Example}
\newtheorem{em-def}[deff]{Definition}        
\newtheorem{em-remark}[deff]{Remark}         
\newtheorem{em-question}[deff]{Question}

\newtheorem{problem}[deff]{Problem}

\theoremstyle{definition}
\newtheorem{definition}[deff]{Definition}

\newenvironment{example}{\begin{em-example} \em }{ \end{em-example}}

\newenvironment{remark}{\begin{em-remark} \em }{\end{em-remark}}

\newcommand{\Ra}{\Rightarrow}

\newcommand{\w}{\omega}

\newcommand{\e}{\varepsilon}
\newcommand{\IR}{\mathbb R}
\newcommand{\IN}{\mathbb N}
\newcommand{\cs}{\mathsf{cs}}

\DeclareMathSymbol{\res}{\mathord}{AMSa}{"16}

\def\:{\nobreak \hskip .1111em\mathpunct {}\nonscript \mkern
   -\thinmuskip {:}\hskip .3333emplus.0555em\relax}

\catcode`\@=12

\def\R{{\mathbb R}}

\begin{document}

\title{$\w^\w$-Base and infinite-dimensional compact sets in locally convex spaces}
\author{Taras Banakh}
\address{Ivan Franko National University of Lviv (Ukraine) and Jan Kochanowski University in Kielce (Poland)}

\email{t.o.banakh@gmail.com}
\author{Jerzy K\c akol}
\address{Faculty of Mathematics and Informatics, A. Mickiewicz University  61-614 Pozna{\'{n}} (Poland), and Institute of Mathematics Czech Academy of Sciences, Prague (Czech Republic)  }

\email{kakol@amu.edu.pl}
\author{Johannes Phillip Sch\"urz}
\address{Faculty of Mathematics and Geoinformation, TU Wien,
 1040 Wien, (Austria)}
\email{johannes.schuerz@tuwien.ac.at}
\begin{abstract}
A locally convex space (lcs) $E$ is said to have an $\omega^{\omega}$-base if $E$ has a neighborhood base $\{U_{\alpha}:\alpha\in\omega^\omega\}$ at zero such that $U_{\beta}\subseteq  U_{\alpha}$ for all $\alpha \leq\beta$. The class of lcs with an $\omega^{\omega}$-base is large, among others contains all $(LM)$-spaces (hence $(LF)$-spaces), strong duals of distinguished Fr\'echet lcs (hence spaces of distributions $D'(\Omega)$).
A remarkable result of Cascales-Orihuela states that every compact set in a lcs with an $\omega^{\omega}$-base is metrizable. Our main result shows that every uncountable-dimensional lcs with an $\omega^{\omega}$-base contains an infinite-dimensional metrizable compact subset.  On the other hand, the countable-dimensional space $\varphi$ endowed with the finest locally convex topology has an $\w^\w$-base but contains no infinite-dimensional compact subsets. It turns out that $\varphi$ is a unique infinite-dimensional locally convex space which is a $k_\IR$-space containing no infinite-dimensional compact subsets. Applications to spaces $C_{p}(X)$ are provided.
\end{abstract}

\thanks{The research for the second  named author is supported  by the GA\v{C}R project 20-22230L and RVO: 67985840.  }

\maketitle

\section{Introduction}

A topological space $X$  is said  to have a neighborhood  $\omega^{\omega}$-base  at a point $x\in X$ if there exists a
neighborhood base $(U_{\alpha}(x))_{\alpha\in\w^{\w}}$  at $x$ such that $U_{\beta}(x)\subseteq U_{\alpha}(x)$ for all  $\alpha\leq\beta$ in $\w^{\w}$. We say
that $X$  has an $\omega^{\omega}$-base if it has a neighborhood  $\omega^{\omega}$-base at each point of $X$. Evidently,
a topological group (particularly topological vector space (tvs))  has an  $\omega^{\omega}$-base if it has a neighborhood $\omega^{\omega}$-base at the identity.
 The classical metrization theorem of Birkhoff and Kakutani states that a
topological group $G$ is metrizable if and only if $G$  is first-countable. Then, as easily seen, if $(U_{n})_{n\in\w}$ is a neighborgood base at the identity of $G$, then the family $\{U_{\alpha}:\alpha\in\w^{\w}\}$ formed by sets $U_{\alpha}= U_{\alpha(0)}$ forms an  $\omega^{\omega}$-base (at the identity) for $G$.  Locally convex spaces (lcs)  with an $\omega^{\omega}$-base  are known in Functional Analysis since 2003 when Cascales, K\c akol, and
Saxon \cite{cas-kak-saxon} characterized quasi-barreled lcs with an $\omega^{\omega}$-base. In several papers (see \cite{kak} and the references therein) spaces with an $\omega^{\omega}$-base
were studied under the name \emph{lcs with a $\mathfrak{G}$-base},  but  here we prefer (as in \cite{banakh}) to use  the more self-suggesting
terminology of $\omega^{\omega}$-bases.

In \cite{CasOri} Cascales and Orihuela proved that compact subsets of any lcs with an $\omega^{\omega}$-base are metrizable. This refers, among others, to each
 $(LM)$-space, i.e. a countable inductive limit of metrizble lcs, since $(LM)$-spaces have an $\omega^{\omega}$-base. Also the following \emph{metrization theorem} holds together a number of topological conditions.

 \begin{theorem}\label{baire-like} \cite[Corollary 15.5]{kak}  For a barrelled lcs $E$  with an $\omega^{\omega}$-base, the following conditions are equivalent.
\begin{enumerate}
\item $E$ is metrizable;
\item  $E$ is Fr\'echet-Urysohn;
\item  $E$ is Baire-like;
\item  $E$ does not contain a   copy of $\varphi$, i.e. an  $\aleph_{0}$-dimensional vector space endowed with the finest locally convex topology.
\end{enumerate}
\end{theorem}
Hence every Baire lcs with an $\omega^{\omega}$-base is metrizable. The space $\varphi$ appearing in Theorem~\ref{baire-like} has the following properties: 
\begin{enumerate}
\item  $\varphi$ is the strong dual of the Fr\'echet-Schwartz space $\mathbb{R}^{\omega}$.
\item  All compact subsets in $\varphi$ are  finite-dimensional.
\item  $\varphi$ is a complete bornological space,
\end{enumerate}
 see \cite{saxon}, \cite{nyikos}, \cite{kak}.

Being motivated by above's results, especially by a remarkable theorem of Cascales-Oruhuela mentioned above, one can ask  for  a possible large class of lcs $E$ for which  every infinite-dimensional subspace of $E$ contains an infinite-dimensional compact (metrizable) subset.  Surely, each  metrizable lcs trivially fulfills this request. We prove however the following general

\begin{theorem}\label{t:main}
Every uncountably-dimensional lcs $E$ with $\omega^{\omega}$-base contains  an infinite-dimensional metrizable compact subset.
\end{theorem}

Theorem~\ref{t:main} will be proved in Section 4. An alternative proof will be presented in Section 5 as a consequence of Theorem \ref{t:main3}.

The uncountable dimensionality of the space $E$  in Theorem~\ref{t:main} cannot be replaced by the infinite-dimensionality of $E$: the space $\varphi$ is infinite-dimensional, has an $\w^\w$-base and contains no infinite-dimensional compact subsets. However, $\varphi$ is a unique locally convex $k_\IR$-space with this property. Recall \cite{kR} that a topological space $X$ is a {\em $k_\IR$-space} if a function $f:X\to\IR$ is continuous whenever for every compact subset $K\subseteq X$ the restriction $f{\restriction}K$ is continuous. We prove the following

\begin{theorem}\label{t:main2} A lcs $E$ is topologically isomorphic to the space $\varphi$ if and only if $E$ is a $k_\IR$-space containing no infinite-dimensional compact subsets.
\end{theorem}

Theorem~\ref{t:main2} implies that a lcs is topologically isomorphic to $\varphi$ if and only if it is homeomorphic to $\varphi$. This topological uniqueness property of the space $\varphi$ was first proved by the first author in \cite{Ban98}.

The following characterization of the space $\varphi$ can be derived from Theorems~\ref{t:main} and \ref{t1}. It shows that $\varphi$ is a unique bornological space for which the uncountable dimensionality in Theorem~\ref{t:main} cannot be weakened to infinite dimensionality.

\begin{theorem} A lcs $E$ is topologically isomorphic to the space $\varphi$ if and only if $E$ is bornological, has an $\w^\w$-base and contains no infinite-dimensional compact subset.
\end{theorem}

Theorem \ref{t:main} provides a large class of concrete (non-metrizable) lcs containing infinite-dimensional compact  sets.

\begin{corollary} \label{com}
Every  uncountable-dimensional subspace of an $(LM)$-space contains an infinite-dimensional compact set.
\end{corollary}
 Let $X$ be a Tychonoff space. By $C_{p}(X)$ and $C_{k}(X)$  we denote the space of continuous real-valued functions on $X$ endowed with the pointwise and the compact-open topology, respectively.
The problem of characterization of  Tychonoff spaces $X$ whose function spaces $C_{p}(X)$ and $C_{k}(X)$ admit an $\w^{\w}$-base is already solved. Indeed, by \cite[Corollary 15.2]{kak} $C_{p}(X)$ has   an $\w^{\w}$-base if and only if $X$ is countable. The space $C_{k}(X)$
has an $\w^{\w}$-base if and only if $X$ admits a fundamental compact resolution \cite{feka}, for necessary definitions see below. Since every \v Cech-complete Lindel\"of space $X$ is a continuous image of a  Polish space under a perfect map (and the latter space admits a fundamental compact resolution), the space $C_{p}(X)$ has an  $\w^{\w}$-base. So, we have another  concrete application of Theorem \ref{t:main}.

\begin{example}
Let $X$ be an infinite \v Cech-complete Lindel\"of space. Then every uncountable-dimensional subspace of $C_{k}(X)$ contains an infinite-dimensional metrizable compact set.
\end{example}



 In Section 2 we show that all (bornological) lcs containing no infinite-dimensional compact subsets are  bornologically (and topologically) isomorphic to a  free lcs over discrete topological spaces. Consequently, in Sections 3 and 4 we study the free lcs $L(\kappa)$  over infinite cardinals $\kappa$, including  $L(\w)=\varphi$. We introduce two concepts: the  $(\kappa,\lambda)$-tall bornology and  the $(\kappa,\lambda)_p$-equiconvergence, which will be used  to obtain  bornological and topological characterizations of  $L(\kappa)$.  Both concepts  apply to prove  Theorem~\ref{t:main}. To this end, we shall prove that each topological (vector) space with an $\w^\w$-base is $(\w_1,\w)_p$-equiconvergent (and has $(\w_1,\w)$-tall bornology). Another property implying the $(\w_1,\w)_{p}$-equiconvergence is the existence of a countable $\cs^\bullet$-network (see Theorem \ref{t:network}), which follows from the existence of an $\w^\w$-base according to Proposition~\ref{p9}. Linear counterparts of $\cs^\bullet$-networks are radial networks introduced in Section~\ref{s:rn}, whose main result is Theorem \ref{t:main3} implying Theorem~\ref{t:main}. Some applications of Theorem~\ref{t:main} to function spaces $C_{p}(X)$  are provided in Section~\ref{s:Cp}.

\section{Locally convex spaces containing no infinite-dimensional compact subsets}

In this section we study lcs containing no infinite-dimensional compact subsets. We shall show that all such (bornological) spaces are bornologically (and topologically) isomorphic to free lcs over discrete topological spaces.

Recall that for a topological space $X$ its {\em free locally convex space} is a lcs $L(X)$ endowed with a continuous function $\delta:X\to L(X)$ such that for any continuous function $f:X\to E$ to a lcs  $E$ there exists a unique linear continuous map $T:L(X)\to E$ such that $T\circ\delta=f$. The set $X$ forms a
Hamel basis for $L(X)$ and $\delta$ is a topological embedding, see \cite{raikov}; we also refer  to \cite{BL} and \cite{banakh} for several results and references concerning this concept; \cite[Theorem 5.4]{BL} characterizes  those $X$ for which  $L(X)$ has an $\w^{\w}$-base.

Let $E$ be a tvs.  A subset $B\subseteq E$ is called {\em bounded} if for every neighborhood $U\subseteq E$ of zero there exists $n\in\IN$ such that $B\subseteq nU$. The family of all bounded sets of $E$ is called the {\em bornology} of $E$.
 A linear operator $f:E\to F$ between two tvs  is called {\em bounded} if for any bounded set $B\subseteq E$ its image $f(B)$ is bounded in $F$.

Two tvs $E$ and $F$ are
\begin{itemize}
\item {\em topologically isomorphic} if there exists a linear bijective function $f:E\to F$ such that $f$ and $f^{-1}$ are continuous;
\item {\em bornologically isomorphic} if  there exists a linear bijective function $f:E\to F$ such that $f$ and $f^{-1}$ are bounded.
\end{itemize}
A lcs $E$ is called  {\em bornological} if each bounded linear operator  from $E$ to a lcs $F$ is continuous.
A linear space $E$ is called {\em $\kappa$-dimensional} if $E$ has a Hamel basis of cardinality $\kappa$. In this case we write $\kappa=\dim(E)$.

A lcs $E$ is {\em free} if it carries the finest locally convex topology.
In this case $E$ is topologically isomorphic to the free lcs $L(\kappa)$ over the cardinal $\kappa=\dim(E)$ endowed with the discrete topology.

The study around  the free lcs $L(\w)=\varphi$  has attracted specialists  for a long time.  For example, Nyikos observed \cite{nyikos}  that each sequentially  closed subset of $L(\w)$ is closed although the sequential closure of a subset of $\varphi$ need not be closed. Consequently, $L(\w)$ is a concrete ``small" space without the Fr\'echet-Urysohn property. Applying the  Baire category theorem one shows that $L(\w)$ is not a Baire-like space (in sense of Saxon \cite{saxon}) and   a barrelled  lcs $E$ is Baire-like if $E$ does not contain a  copy of $L(\w)$, see \cite{saxon}. Although $L(\w)$ is not Fr\'echet-Urysohn, it provides some extra properties since all vector subspaces in $L(\w)$ are closed. In \cite{kakol-saxon} we introduced  the property for a lcs $E$ (under the name $C_{3}^{-}$) stating that the sequential closure of every linear subspace of $E$ is sequentially closed, and we proved \cite[Corollary 6.4]{kakol-saxon} that the only infinite-dimensional Montel (DF)-space with property  $C_{3}^{-}$ is $L(\w)$ (yielding a remarkable result of Bonet and Defant that the only infinite-dimensional Silva space with property $C_{3}^{-}$ is $L(\w)$). This implies that  barrelled $(DF)$-spaces and $(LF)$-spaces satisfying  property $C_{3}^{-}$ are exactly  of the form $M$, $L(\w)$, or
$M\times L(\w)$ where $M$ is metrizable, \cite[Theorems 6.11, 6.13]{kakol-saxon}.



The following simple theorem characterizes lcs containing no infinite-dimensional compact subsets.

\begin{theorem}\label{t1} For a lcs $E$ the following conditions are equivalent:
\begin{enumerate}
\item Each compact subset of $E$  has finite topological dimension.
\item Each bounded linearly independent set in $E$ is finite.
\item $E$ is bornologically isomorphic to a free lcs.
\end{enumerate}
If $E$ is bornological, then the conditions \textup{(1)--(3)} are equivalent to
\begin{itemize}
\item[(4)] $E$ is free.
\end{itemize}
\end{theorem}

\begin{proof} $(1)\Ra(2)$ Suppose that each compact subset of $E$ has finite topological dimension. Assuming that $E$ contains an infinite bounded linearly indendent set, we can find a bounded linearly independent set $\{x_n\}_{n\in\w}$ consisting of pairwise distinct points $x_n$. Then the sequence $(2^{-n}x_n)_{n\in\w}$ converges to zero and $$K=\textstyle\bigcup_{n\in\w}\big\{\sum\limits_{k=n}^{2n}t_kx_k:(t_k)_{k=n}^{2n}\in\prod\limits_{k=n}^{2n}[0,2^{-k}]\big\}$$is an infinite-dimensional compact set in $E$, which contradicts our assumption.
\smallskip

$(2)\Ra(3)$ Let $\tau$ be the finest locally convex topology on  $E$. Then the identity map $(E,\tau)\to E$ is continuous and hence bounded. If each bounded linearly independent set in $E$ is finite, then each bounded set $B\subseteq E$ is contained in a finite-dimensional subspace of $E$ and hence is bounded in the topology $\tau$. This means that the identity map $E\to(E,\tau)$ is bounded and hence $E$ is bornologically isomorphic to the free lcs $(E,\tau)$.
\smallskip

$(3)\Ra(1)$ If $E$ is bornologically isomorphic to a free lcs $F$ then each bounded linearly independent set in $E$ is finite, since the free lcs $F$ has this property.
\smallskip

The implication $(4)\Ra(3)$ is trivial. If $E$ is bornological then the implication $(3)\Ra(4)$ follows from the continuity of bounded linear operators on bornological spaces.
\end{proof}

The free  lcs over discrete topological spaces are not unique lcs possessing no infinite-dimensional compact sets.
A subset $B$ of a topological space $X$ is called {\em functionally bounded} if for any continuous real-valued function $f:X\to \IR$ the set $f(B)$ is bounded.

\begin{proposition} For a Tychonoff space $X$ the following conditions are equivalent:
\begin{enumerate}
\item each compact subset of the free lcs $L(X)$ has finite topological dimension;
\item each bounded linearly independent set in $L(X)$ is finite;
\item each functionally bounded subset of $X$ is finite.
\end{enumerate}
\end{proposition}
\begin{proof} The equivalence $(1)\Leftrightarrow(2)$ follows from the corresponding equivalence in Theorem~\ref{t1}.
The implication $(3)  \Rightarrow (1)$ follows from  \cite[Lemma 10.11.3]{Ban-Fan}, and $(2)   \Rightarrow (3)$ follows from the observation that each functionally bounded set in a lcs is bounded.
\end{proof}

\section{Bornological and topological characterizations of the spaces $L(\kappa)$}

In this section, given an infinite cardinal $\kappa$ we characterize the free lcs $L(\kappa)$ using some specific properties of the bornology and the topology of the space $L(\kappa)$.

Let $\kappa,\lambda$ be two cardinals. A lcs $E$ is defined to have {\em $(\kappa,\lambda)$-tall bornology} if every subset $A\subseteq E$ of cardinality $|A|=\kappa$ contains a bounded subset $B\subseteq A$ of cardinality $|A|=\lambda$.
\begin{theorem}\label{t8} Let $\kappa$ be an infinite cardinal.
For a lcs $E$ the following conditions are equivalent:
\begin{enumerate}
\item $E$ is bornologically isomorphic to the free lcs $L(\kappa)$;
\item each bounded linearly independent set in $E$ is finite and the bornology of $E$ is $(\kappa^+,\w)$-tall but not $(\kappa,\w)$-tall.
\end{enumerate}
If $E$ is bornological, then the conditions \textup{(1)--(2)} are equivalent to
\begin{itemize}
\item[(3)] $E$ is topologically isomorphic  to $L(\kappa)$.
\end{itemize}
\end{theorem}

\begin{proof} $(1)  \Rightarrow(2)$: Assume that $E$ is bornologically isomorphic to $L(\kappa)$. Then $E$ has algebraic dimension $\kappa$ and each bounded linearly independent set in $E$ is finite (since this is true in $L(\kappa)$).

To see that the bornology of $E$ is $(\kappa^+,\w)$-tall, take any set $K\subseteq E$ of cardinality $|K|=\kappa^+$. Since  $E$ has algebraic dimension $\kappa$, there exists a cover $(B_\alpha)_{\alpha\in\kappa}$ of $E$ by $\kappa$ many compact sets. By the Pigeonhole Principle, there exists $\alpha\in\kappa$ such that $|K\cap B_\alpha|=\kappa^+$. This means that the bornology of $E$ is $(\kappa^+,\kappa^+)$-tall and hence $(\kappa^+,\w)$-tall.

To see that the bornology of the space $E$ is not $(\kappa,\w)$-tall, observe that the  Hamel basis $\kappa$ of $L(\kappa)$ has the property that no infinite subset of $\kappa$ is bounded in $L(\kappa)$. Since $E$ is bornologically  isomorphic to $L(\kappa)$, the image of $\kappa$ in $E$ is a subset of cardinality $\kappa$ containing no bounded infinite subsets and witnessing that $E$ is not $(\kappa,\w)$-tall.

$(2) \Rightarrow (1)$:  Assume that each bounded linearly independent set in $E$ is finite and the bornology of $E$   is $(\kappa^+,\w)$-tall but not $(\kappa,\w)$-tall. Let $B$ be a Hamel basis of $E$. We claim that $|B|=\kappa$. Assuming that $|B|>\kappa$, we conclude that $E$ is not $(\kappa^+,\w)$-tall, which is a contradiction. Assuming that $|B|<\kappa$, we conclude that $E$ is the union of $<\kappa$ many bounded sets and hence is $(\kappa,\kappa)$-tall by the Pigeonhole Principle. But this contradicts our assumption. Therefore $|B|=\kappa$. Let $h:\kappa\to B$ be any  bijection and $\bar h:L(\kappa)\to E$ be the unique extension of $h$ to a linear continuous operator. Since $B$ is a Hamel basis for $E$, the operator $\bar h$ is bijective. Since each bounded set in $E$ is contained in a finite-dimensional linear subspace, the operator $\bar h^{-1}:E\to L(\kappa)$ is bounded and hence $\bar h:L(\kappa)\to E$ is a bornological  isomorphism.

If the space $E$ is bornological, then the equivalence $(1)\Leftrightarrow(3)$ follows from the bornological property of $E$ and $L(\kappa)$.
\end{proof}

The $(\kappa,\w)$-tallness of the bornology of a lcs $E$  has topological counterparts introduced in the following definition.
\begin{definition} Let $\kappa,\lambda$ be cardinals. We say that a topological space $X$ is
\begin{itemize}
\item {\em $(\kappa,\lambda)_p$-equiconvergent at a point} $x\in X$ if for any  indexed family $\{x_\alpha\}_{\alpha\in\kappa}\subseteq \{s\in X^\w:\lim_{n\to\infty}s(n)=x\}$, there exists a subset $\Lambda\subseteq\kappa$ of cardinality $|\Lambda|=\lambda$ such that for every neighborhood $O_x\subseteq X$ of $x$ there exists $n\in\w$ such that the set $\{\alpha\in\Lambda:x_\alpha(n)\notin O_x\}$ is finite;
\item {\em $(\kappa,\lambda)_k$-equiconvergent at a point} $x\in X$ if for any  indexed family $\{x_\alpha\}_{\alpha\in\kappa}\subseteq \{s\in X^\w:\lim_{n\to\infty}s(n)=x\}$, there exists a subset $\Lambda\subseteq\kappa$ of cardinality $|\Lambda|=\lambda$ such that for every neighborhood $O_x\subseteq X$ of $x$ there exists $n\in\w$ such that for every $m\ge n$ and $\alpha\in\Lambda$ we have  $x_\alpha(m)\in O_x$;
\item {\em $(\kappa,\lambda)_p$-equiconvergent} if $X$ is $(\kappa,\lambda)_p$-equiconvergent at every point $x\in X$;
\item {\em $(\kappa,\lambda)_k$-equiconvergent} if $X$ is $(\kappa,\lambda)_k$-equiconvergent at every point $x\in X$.
\end{itemize}
\end{definition}
It is easy to see that every $(\kappa,\lambda)_k$-equiconvergent space is  $(\kappa,\lambda)_p$-equiconvergent.
The following observation will be used  below.

\begin{proposition}\label{p9} If a lcs $E$ is $(\kappa,\lambda)_p$-equiconvergent, then its bornology is $(\kappa,\lambda)$-tall.
\end{proposition}

\begin{proof} Given a subset $K\subseteq E$ of cardinality $|K|=\kappa$, for every $\alpha\in K$ consider the convergent sequence $x_\alpha\in X^\w$ defined by $x_\alpha(n)=2^{-n}\alpha$. Assuming that the lcs $E$ is  $(\kappa,\lambda)_p$-equiconvergent, we can find a subset $L\subseteq K$  of cardinality $|L|=\lambda$ such that for every neighborhood of zero $U\subseteq E$ there exists $n\in\w$ such that the set $\{\alpha\in L:2^{-n}\alpha\notin U\}$ is finite. We claim that the set $L$ is bounded. Indeed, for every neighborhood $U\subseteq E$ of zero, we find a neighborhood $V\subseteq E$ of zero such that $[0,1]\cdot V\subseteq U$. By our assumption, there exists $n\in\w$ such that the set $F=\{\alpha\in K:2^{-n}\alpha\notin V\}$ is finite. Find $m\ge n$ such that $2^{-m}\alpha\in U$ for every $\alpha\in F$.  Then  $2^{-m}L\subseteq 2^{-m}(L\setminus F)\cup 2^{-m}F\subseteq ([0,1]\cdot V)\cup U=U$, and hence the set $L$ is bounded.
\end{proof}
Nevertheless, it seems that the following question  remains open.
\begin{problem} Assume that the bornology of a lcs $E$ is $(\w_1,\w)$-tall. Is it true that  $E$ is $(\w_1,\w)_p$-equiconvergent?
\end{problem}
Below we prove the following topological counterpart to  Theorem \ref{t8}.

\begin{theorem} \label{tlast} Let $\kappa$ be an infinite cardinal.
 For a lcs $E$ the following conditions are equivalent:
\begin{enumerate}
\item $E$ is bornologically isomorphic to $L(\kappa)$;
\item each compact subset of $E$ has finite topological dimension, $E$ is $(\kappa^+,\w)_k$-equiconvergent but not $(\kappa,\w)_p$-equiconvergent.
\item each compact subset of $E$ has finite topological dimension, $E$ is $(\kappa^+,\w)_p$-equiconvergent but not $(\kappa,\w)_k$-equiconvergent.
\end{enumerate}
If $E$ is bornological, then the conditions \textup{(1)--(3)} are equivalent to
\begin{itemize}
\item[(4)] $E$ is topologically isomorphic to $L(\kappa)$.
\end{itemize}
\end{theorem}

\begin{proof} $(1)   \Rightarrow (2)$:  Assume that $E$ is bornologically isomorphic to $L(\kappa)$. By Theorems~\ref{t8} each bounded linearly independent set in $E$ is finite, and by Theorem~\ref{t1}, each compact subset of $E$  is finite-dimensional.  The linear space $E$ has algebraic dimension $\kappa$, being isomorphic to the linear space $L(\kappa)$. Let $B$ be a Hamel basis for the space $E$.

 To show that $E$ is $(\kappa^+,\w)_k$-equiconvergent, fix an indexed family $\{x_\alpha\}_{\alpha\in\kappa^+}\subseteq\{s\in E^\w:\lim_{n\to\infty}s(n)=0\}.$ Since bounded linearly independent sets in $E$ are finite, for every $\alpha\in\kappa^+$ there exists a finite set $F_\alpha\subseteq B$ such that the bounded set $x_\alpha[\w]$ is contained in the linear hull of $F_\alpha$. Since $|B|=\kappa<\kappa^+$, by the Pigeonhole Principle, for some finite set $F\subseteq B$ the set $A=\{\alpha\in\kappa^+:F_\alpha=F\}$ is uncountable. Let $[F]$ be the linear hull of the finite set $F$ in the linear space $E$.

Consider the ordinal $\w+1=\w\cup\{\w\}$ endowed with the compact metrizable topology generated by the linear order.  For every $\alpha\in A$ let $\bar x_\alpha:\w+1\to [F]$ be the continuous function such that $\bar x_\alpha{\restriction}\w=x_\alpha$ and $\bar x_\alpha(\w)=0$. Let $C_k(\w+1,[F])$ be the space of continuous functions from $\w+1$ to $[F]$, endowed with the compact-open topology. Since $A$ is uncountable and the space $C_k(\w+1,[F])\supseteq\{\bar x_\alpha\}_{\alpha\in A}$ is Polish, there exists a sequence $\{\alpha_n\}_{n\in\w}\subseteq A$ of pairwise distinct ordinals such that the sequence $(\bar x_{\alpha_n})_{n\in\w}$ converges to $\bar x_{\alpha_0}$ in the function space $C_k(\w+1,[F])$. Then the set $\Lambda=\{\alpha_n\}_{n\in\w}\subseteq\kappa^+$ witnesses that $E$ is $(\kappa^+,\w)_k$-equiconvergent to zero and by the topological homogeneity, $E$ is $(\kappa^+,\w)$-equiconvergent.
By Theorem~\ref{t8}, the bornology of the space $E$ is not $(\kappa,\w)$-tall. By Proposition~\ref{p9}, the space $E$ is not $(\kappa,\w)_p$-equiconvergent.

The implication $(2)   \Rightarrow (3)$ is trivial.
To prove that $(3) \Rightarrow  (1)$, assume that each compact subset of $E$ has finite topological dimension and $E$ is $(\kappa^+,\w)_p$-equiconvergent but not $(\kappa,\w)_k$-equiconvergent. Let $B$  be a Hamel basis in $E$. By Theorem~\ref{t1}, the space $E$ is bornologically isomorphic to $L(|B|)$. Applying the (already proved) implication $(1)   \Rightarrow (2)$, we conclude that $E$ is $(|B|^+,\w)_k$-equiconvergent, which implies that $|B|\ge\kappa$ (as $E$ is not $(\kappa,\w)_k$-equiconvergent). Assuming that $|B|>\kappa$, we can see that the family $\{x_b\}_{b\in B}\subseteq E^\w$ of the sequences $x_b(n)=2^{-n}b$ witnesses that $E$ is not $(|B|,\w)_p$-equiconvergent and hence not $(\kappa^+,\w)_p$-equiconvergent, which contradicts our assumption. So, $|B|=\kappa$ and $E$ is bornologically isomorphic to $L(\kappa)$.
If the space $E$ is bornological, then the equivalence $(1)\Leftrightarrow(4)$ follows from the bornological property of $E$ and $L(\kappa)$.
 \end{proof}

Observe that the purely topological properties (2), (3) in Theorem~\ref{tlast} characterize the free lcs $L(\kappa)$ up to bornological equivalence. We do not know whether the topological structure of the space $L(\kappa)$ determines this lcs uniquely up to a topological isomorphism.

\begin{problem} Assume that a lcs $E$ is homeomorphic to the free lcs $L(\kappa)$ for some cardinal $\kappa$. Is $E$ topologically isomorphic to $L(\kappa)$?
\end{problem}
By \cite{Ban98} the answer to this problem is affirmative for $\kappa=\w$. This affirmative answer can also be derived from the following topological characterizations of the space $L(\w)=\varphi$. This characterization has been announced in the introduction as Theorem~\ref{t:main2}.
\begin{theorem} A lcs $E$ is topologically isomorphic to the free lcs $L(\w)$ if and only if $E$ is an infinite-dimensional $k_\IR$-space containing no infinite-dimensional compact subset.
\end{theorem}
\begin{proof} The ``only if'' part follows from known topological properties of the space $L(\w)=\varphi$ mentioned in the introduction. To prove the ``if'' part, assume that a lcs $E$ is a $k_\IR$-space and each  compact subset of $E$ is finite-dimensional. Choose a Hamel basis $B$ in $E$ and consider the linear continuous operator $T:L(B)\to E$ such that $T(b)=b$ for each $b\in B$. Since $B$ is a Hamel basis, the operator $T$ is injective. We claim that the operator $T^{-1}:E\to L(B)$ is bounded. By Theorem \ref{t1}   the linear hull of each compact subset  $K\subseteq E$  is finite-dimensional, which implies that the restriction $T^{-1}{\restriction}K$ is continuous. Since $E$ is a $k_\IR$-space, $T^{-1}$ is continuous and hence $T$ is a topological isomorphism. Then the free lcs $L(B)$ is a $k_\IR$-space. Applying \cite{Gab}, we conclude that $B$ is countable and hence $E$ is topologically isomorphic to $L(\w)$.
\end{proof}

A Tychonoff space $X$ is called {\em Ascoli} if the canonical map $\delta:X\to C_k(C_k(X))$ assigning to each point $x\in X$ the Dirac functional $\delta_x:C_k(X)\to \IR$, $\delta_x:f\mapsto f(x)$, is continuous.  By \cite{BG}, the class of Ascoli spaces includes all Tychonoff $k_\IR$-spaces. By \cite{Gab} a Tychonoff space $X$ is countable and discrete if and only if its free lcs $L(X)$ is Ascoli.

\begin{problem} Assume that an infinite-dimensional lcs $E$ is Ascoli and contains no infinite-dimensional compact subsets. Is $E$ topologically isomorphic to the space $L(\w)$?
\end{problem}

\section{Equiconvergence of topological spaces and  proof of Theorem \ref{t:main}}\label{s:network}

In this section we establish two results related to equiconvergence in topological spaces.

\begin{theorem}\label{follows}  If a topological space $X$ admits  an $\w^\w$-base at a point $x\in X$, then $X$ is $(\w_1,\w)_k$-equiconvergent at the point $x$.
\end{theorem}
\begin{proof} Let $(U_f)_{f\in\w^\w}$ be an $\w^\w$-base at $x$. To show that $X$ is $(\w_1,\w)_k$-equiconvergent at $x$, fix an indexed family $$\{x_\alpha\}_{\alpha\in\w_1}\subseteq\{s\in X^\w:\lim_{n\to\infty}s(n)=x\}$$ of sequences that converge to $x$. For every $\alpha\in\w_1$ consider the function $\mu_\alpha:\w^\w\to\w$ assigning to each $f\in\w^\w$ the smallest number $n\in\w$ such that $\{x_\alpha(m)\}_{m\ge n}\subseteq U_f.$ It is easy to see that the function $\mu_\alpha:\w^\w\to\w$ is monotone.

For every $n\in\w$ and finite function $t\in\w^n$, let $\w^\w_t=\{f\in\w^\w:f{\restriction}n=t\}.$ By  \cite[Lemma 2.3.5]{banakh}, for every $f\in\w^\w$ there exists $n\in\w$ such that $\mu_\alpha[\w^\w_{f{\restriction}n}]$ is finite. Let $T_\alpha$ be the set of all finite functions $t\in\w^{<\w}=\bigcup_{n\in\w}\w^n$ such that $\mu_\alpha[\w^\w_t]$ is finite but for any $\tau\in\w^{<\w}$ with $\tau\subset t$ the set $\mu_\alpha[\w^\w_\tau]$ is infinite. It follows from  \cite[Lemma 2.3.5]{banakh} that  for every $f\in\w^\w$ there exists a unique $t_f\in T_{\alpha}$ such that $t_f\subset  f$.

Let $\delta_\alpha(f)=\max\mu_\alpha[\w^\w_{t_f}]\ge \mu_\alpha(f)$. It is clear that the function $\delta_\alpha:\w^\w\to\w$ is continuous and hence $\delta_\alpha$ is an element of the space $C_p(\w^\w,\w)$ of continuous functions from $\w^\w$ to $\w$. Here we endow $\w^\w$ with the product topology.
The function space $C_p(\w^\w,\w)$ is endowed with the topology of poitwise convergence. By Michael's Proposition 10.4 in \cite{michael},  the  space $C_p(\w^\w,\w)$ has a countable network.

Consider the function $\delta:\w_1\to C_p(\w^\w,\w), \,\,\,\delta:\alpha\mapsto\delta_\alpha,$ and observe that $\delta_\alpha(f)\ge \mu_\alpha(f)$ for any $\alpha\in\w_1$ and $f\in\w^\w$.

Since the space $C_p(\w^\w,\w)$ has countable network, there exists a  sequence $\{\alpha_n\}_{n\in\w}\subseteq \omega_1$ of pairwise distinct ordinals such that the sequence $(\delta_{\alpha_n})_{n\in\w}$ converges to $\delta_{\alpha_0}$ in the function space $C_p(\w^\w,\w)$.
We claim that the sequence $(x_{\alpha_n})_{n\in\w}$ witnesses that $X$ is $(\w_1,\w)_k$-equiconvergent at $x$. Given any open neighborhood $O_x\subseteq X$ of $x$, find $f\in\w^\w$ such that $U_f\subseteq O_x$. Since the sequence $(x_{\alpha_0}(n))_{n\in\w}$ converges to $x$, there exists $m\in\w$ such that $\{x_{\alpha_0}(n)\}_{n\ge m}\subseteq U_f$. Since the sequence $(\delta_{\alpha_n})_{n\in\w}$ converges to $\delta_{\alpha_0}$ in $C_p(\w^\w,\w)$ we can replace $m$ by a larger number and additionally assume that $\delta_{\alpha_n}(f)=\delta_{\alpha_0}(f)$ for all $n\ge m$. Choose a number $l\ge\delta_{\alpha_0}(f)$ such that for every $n<m$ and $k\ge l$ we have $x_{\alpha_n}(k)\in O_x$. On the other hand, for every $n\ge m$ and $k\ge l$ we have $k\ge l\ge \delta_{\alpha_0}(f)=\delta_{\alpha_n}(f)\ge\mu_{\alpha_n}(f)$ and hence  $x_{\alpha_n}(k)\in U_f\subseteq O_x$.
\end{proof}

Another property implying the $(\w_1,\w)_{p}$-equiconvergence is the existence of a countable $\cs^\bullet$-network. First we introduce the necessary definitions.

Let $x$ be a point of a topological space $X$. We say that a sequence $\{x_n\}_{n\in\w}\subseteq X$ {\em accumulates at} $x$ if for each neighborhood $U\subseteq X$ of $x$ the set $\{n\in\w:x_n\in U\}$ is infinite.

A family $\mathcal N$ of subsets of  $X$ is defined to be
\begin{itemize}
\item an {\em $\mathsf s^*$-network at $x$} if for any neighborhood $O_x\subseteq X$ of $x$ and any sequence $\{x_n\}_{n\in\w}\subseteq X$ that accumulates at $x$ there exists $N\in\mathcal N$ such that $N\subseteq O_x$ and the set $\{n\in\w:x_n\in N\}$ is infinite;
\item a {\em $\mathsf{cs}^*$-network at $x\in X$} if for any neighborhood $O_x\subseteq X$ of $x$ and any sequence $\{x_n\}_{n\in\w}\subseteq X$ that converges to $x$ there exists $N\in\mathcal N$ such that $N\subseteq O_x$ and the set $\{n\in\w:x_n\in N\}$ is infinite;
\item a {\em $\mathsf{cs}^\bullet$-network at $x$} if for any neighborhood $O_x\subseteq X$ of $x$ and any sequence $\{x_n\}_{n\in\w}\subseteq X$ that converges to $x$ there exists $N\in\mathcal N$ such that $N\subseteq O_x$ and $N$ contains some point $x_n$.
\item a {\em network at $x$} if for any neighborhood $O_x\subseteq X$ the union $\bigcup\{N\in\mathcal N:N\subseteq O_x\}$ is a neighborhood of $x$;
\end{itemize}
It is clear that for any family $\mathcal N$ of subsets of a topological space $X$ and any $x\in X$ we have the following implications.
$$
\xymatrix{
\mbox{($\mathcal N$ is an $\mathsf s^*$-network at $x$)}\ar@{=>}[d]&\mbox{($\mathcal N$ is a network at $x$)}\ar@{=>}[d]\\
\mbox{($\mathcal N$ is a $\cs^*$-network at $x$)}\ar@{=>}[r]&\mbox{($\mathcal N$ is a $\cs^\bullet$-network at $x$)}
}
$$

\begin{theorem}\label{t:network} If a topological space $X$ has a countable $\cs^\bullet$-network at a point $x\in X$, then $X$ is $(\w_1,\w)_p$-equiconvergent at $x$.
\end{theorem}

\begin{proof} Let $\mathcal N$ be a countable $\cs^\bullet$-network at $x$ and $$\{x_\alpha\}_{\alpha\in\w_1}\subseteq\{s\in X^\w:\lim_{n\to\infty}s(n)=x\}.$$ Endow the ordinal $\w+1=\w\cup\{\w\}$ with the discrete topology.
For every $\alpha\in \omega_1$ consider the function $\delta_\alpha:\mathcal N\to\w+1$ assigning to each $N\in\mathcal N$ the smallest number $n\in\w$ such that $x_\alpha(n)\in N$ if such number $n$ exists, and $\w$ if $x_n\notin N$ for all $n\in\w$.
Since $(\w+1)^{\mathcal N}$ is a metrizable separable space, the uncountable set $$\{\delta_\alpha\}_{\alpha\in \omega_1}\subseteq(\w+1)^{\mathcal N}$$ contains a non-trivial convergent sequence. Consequently, we can find a sequence $(\alpha_n)_{n\in\w}$ of pairwise distinct countable ordinals such that the sequence $(\delta_{\alpha_n})_{n\in\w}$ converges to $\delta_{\alpha_0}$ in the Polish space $(\w+1)^{\mathcal N}$.
We claim that the sequence $(x_{\alpha_n})_{n\in\w}$ witnesses that the space $X$ is $(\w_1,\w)_p$-equiconvergent. Fix any neighborhood $U\subseteq X$ of zero.

Since $\mathcal N$ is an $\cs^\bullet$-network, there exists $N\in\mathcal N$ and $n\in\w$ such that $x_n\in N\subseteq U$. Hence  $$d:=\delta_{\alpha_0}(N)\le n.$$ Since the sequence $(\delta_{\alpha_n})_{n\in\w}$ converges to $\delta_{\alpha_0}$, there exists $l\in\w$ such that $$\delta_{\alpha_k}(N)=\delta_{\alpha_0}(N)=d$$ for all $k\ge l$. Then for every $k\ge l$ we have
$x_{\alpha_k}(d)\in N\subseteq U$.
\end{proof}

The following  proposition (connecting $\w^\w$-bases with networks) is a corollary  of Theorem 6.4.1 in   \cite{banakh}.

 \begin{proposition}\label{p:bn} If $(U_\alpha)_{\alpha\in\w^\w}$ is an $\w^\w$-base at a point $x$ of a topological space $X$, then $(\bigcap_{\beta\in{\uparrow}\alpha}U_\beta)_{\alpha\in\w^{<\w}}$ is a countable $\mathsf s^*$-network at $x$. Here ${\uparrow}\alpha=\{\beta\in\w^\w:\alpha\subset \beta\}$ for any $\alpha\in\w^{<\w}=\bigcup_{n\in\w}\w^n$.
 \end{proposition}

As a consequence of  the results presented above  about  the $(\kappa,\lambda)_p$-equiconvergence   and the $(\kappa,\lambda)$-tall bornology for a lcs $E$,   we propose the following proof of Theorem \ref{t:main}.
\begin{proof}[Proof of Theorem \ref{t:main}]
If a lcs $E$ has an $\w^\w$-base, then by Theorem~\ref{follows},  the space $E$ is $(\w_1,\w)_k$-equiconvergent and hence $(\w_1,\w)_p$-equiconvergent. The $(\w_1,\w)_p$-equiconvergence of $E$ also follows from Proposition~\ref{p:bn} and Theorem~\ref{t:network}. Next, by Proposition~\ref{p9},  the space $E$ has $(\w_1,\w)$-tall bornology, which means that each uncountable set in $E$ contains an infinite bounded set. If $E$ has an uncountable Hamel basis $H$, then $H$ contains an infinite bounded linearly independent set, and by Theorem~\ref{t1} the space $E$ contains an infinite-dimensional compact set.
\end{proof}

\section{Radial networks  and another proof of Theorem~\ref{t:main}}\label{s:rn}

A family $\mathcal N$ of subsets of a linear topological space $E$ is called a {\em radial network} if for every neighborhood of zero $U\subseteq E$ and every every $x\in E$ there exist a set $N\in\mathcal N$ and a nonzero real number $\e$ such that $\e\cdot x\in N\subseteq U$.

The following theorem is a ``linear'' modification of Theorem~\ref{t:network}.

\begin{theorem}\label{t:rn} If a lcs $E$ has a countable radial network, then each uncountable subset in $E$ contains an infinite bounded subset.
\end{theorem}

\begin{proof}  Let $\mathcal N$ be a countable radial network in $E$, and let $A$ be an uncountable set in $E$. Endow the ordinal $\w+1=\w\cup\{\w\}$ with the discrete topology.

For every $\alpha\in A$ consider the function $\delta_\alpha:\mathcal N\to\w+1$ assigning to each $N\in\mathcal N$ the ordinal
$$\delta_\alpha(N)=\min\{n\in\w+1:2^{-n}\cdot\alpha\in [-1,1]\cdot N\}.$$
Here we assume that $2^{-\w}=0$.

Since $(\w+1)^{\mathcal N}$ is a metrizable separable space, the uncountable set $\{\delta_\alpha\}_{\alpha\in A}\subseteq(\w+1)^{\mathcal N}$ contains a non-trivial convergent sequence. Consequently, we can find a sequence $\{\alpha_n\}_{n\in\w}\subseteq A$ of pairwise distinct points of $A$ such that the sequence $(\delta _{\alpha_n})_{n\in\w}$ converges to $\delta_{\alpha_0}$ in the Polish space $(\w+1)^{\mathcal N}$.

We claim that the set $\{\alpha_n\}_{n\in\w}$ is bounded in $X$. Fix any neighborhood $U\subseteq X$ of zero.

Since $\mathcal N$ is a radial network, there exist a set $N\in\mathcal N$ and a nonzero real number $\e$ such that $\e\cdot\alpha_0\in N\subseteq U$. Then $d:=\delta_{\alpha_0}(N)\in\w$. Since the sequence $(\delta_{\alpha_n})_{n\in\w}$ converges to $\delta_{\alpha_0}$, there exists $l\in\w$ such that $\delta_{\alpha_k}(N)=\delta_{\alpha_0}(N)$ for all $k\ge l$. Then for every $k\ge l$ we have
$$2^{-d}\cdot\alpha_k\in [-1,1]\cdot N\subseteq[-1,1]\cdot U$$ and hence
$\{\alpha_k\}_{k\ge l}\subseteq [-2^d,2^d]\cdot U$, which implies that the family $(\alpha_n)_{n\in\w}$ is bounded in $X$.
\end{proof}

The implication $(1)\Ra(7)$ in the following theorem provides an alternative proof of  Theorem~\ref{t:main}, announced in the introduction.

\begin{theorem}\label{t:main3} For a lcs $E$ consider the following properties:
\begin{enumerate}
\item $E$ has an $\w^\w$-base;
\item $E$ has a countable $\mathsf s^*$-network at zero;
\item $E$ has a countable $\cs^*$-network at zero;
\item $E$ has a countable $\cs^\bullet$-network at zero;
\item $E$ has a countable radial network at zero;
\item each uncountable set in $E$ contains an infinite bounded subset;
\item $E$ contains an infinite-dimensional compact set.
\end{enumerate}
Then $(1)\Ra(2)\Ra(3)\Ra(4)\Ra(5)\Ra(6)$. If $E$ has uncountable Hamel basis, then $(6)\Ra(7)$.
\end{theorem}

\begin{proof} The implication $(1)\Ra(2)$ follows from Proposition~\ref{p:bn}. The implications $(2)\Ra(3)\Ra(4)$ are trivial and $(4)\Ra(5)$ follows from the observation that every $\cs^\bullet$-network at zero in the space $E$ is a radial network for $E$. The implication $(5)\Ra(6)$ is proved by  Theorem~\ref{t:rn}.

If $E$ has an uncountable Hamel basis $H$, then by (6), there exists an infinite bounded set $B\subseteq H$. By Theorem~\ref{t1}, the space $E$ contains an infinite-dimensional compact set.
\end{proof}

\begin{problem} Is there an lcs $E$ that has a countable radial network but does not have a countable $\cs^\bullet$-network at zero?
\end{problem}


 \section{Applications to spaces $C_{p}(X)$}\label{s:Cp}

A family $\{B_{\alpha}:\alpha\in\omega^\omega\}$  of bounded (compact) sets  covering  a lcs $E$  is called a \emph{bounded (compact) resolution}    if $B_{\alpha}\subseteq  B_{\beta}$ for each $\alpha\leq\beta$. If additionally every bounded (compact)  subset of $E$ is contained in some $B_{\alpha}$,  we call the family $\{B_{\alpha}:\alpha\in\omega^\omega\}$  a \emph{fundamental bounded (compact) resolution} of $E$.
\smallskip

\begin{example} Let $E$ be a  metrizable  lcs  with a decreasing countable base $(U_{n})_{n\in\w}$ of absolutely convex neighbourhoods of zero. For $\alpha=(n_{k})_{k\in\w}\in\omega^{\omega}$ put
$B_{\alpha}=\bigcap_{k\in\w}n_{k}U_{k}$ and observe that  $\{B_{\alpha}: \alpha\in\omega^{\omega}\}$  is a fundamental bounded resolution in $E$.
\end{example}

 A Tychonoff space $X$ is called {\em pseudocompact} if each continuous real-valued function on $X$ is bounded.

The first part of the following (motivating) result has been proved in \cite{KSMZ}; since this is  not published yet, we add  a short proof.
\begin{proposition}\label{pseudocom}
For a Tychonoff space $X$ the following assertions are equivalent:
\begin{enumerate}
\item The space  $C_{k}(X)$ is covered by a sequence of bounded sets.
\item The space $C_{p}(X)$ is covered by a sequence of bounded sets.
\item   $X$  is pseudocompact.
\end{enumerate}
Moreover, the following assertions are equivalent:
\begin{enumerate}
\item[(4)] $C_{p}(X)$  is covered by a sequence of bounded sets but is  not covered by a sequence of functionally bounded sets.
\item[(5)]   $X$ is pseudocompact and contains a countable  subset which is not closed in $X$ or not $C^{*}$-embedded in $X$.
\end{enumerate}
\end{proposition}
\begin{proof}
(1) $\Rightarrow$ (2) is clear.
(2) $\Rightarrow$ (3): Assume $C_{p}(X)$ is  covered by a sequence of bounded sets  but  $X$ is not psudocompact. Then  $C_{p}(X)$ contains a complemented copy of $\mathbb{R}^{\omega}$, see \cite{arkh4}. But  $\mathbb{R}^{\omega}$ cannot be covered by a sequence of bounded sets, otherwise would be $\sigma$-compact.
(3) $\Rightarrow$ (1): If $X$ is pseudocompact, then for every $n\in\IN$ the set $B_n=\{f\in C(X):\sup_{x\in X} |f(x)|\leq n\}$ is bounded in $C_k(X)$ and $\bigcup_{n\in\IN}B_n=C_k(X)$.
\smallskip

  The equivalence $(4)\Leftrightarrow(5)$ follows from \cite[Problem 399]{Tka1}:  $C_{p}(X)$ is covered by a sequence of functionally bounded subsets o $C_{p}(X)$ if and only if  $X$ is pseudocompact and every countable subset   of $X$ is closed and $C^{*}$-embedded in $X$. \end{proof}
\begin{example}
 $C_{p}([0,\omega_{1}))$ is covered by a sequence of bounded sets but is not covered by a sequence of functionally bounded sets.
\end{example}

By \cite{ferrando}, $C_{p}(X)$ has a bounded resolution if and only if there exists a $K$-analytic space $L$ such that $C_{p}(X)\subseteq L\subseteq\mathbb{R}^{X}$.
The problem  when $C_{p}(X)$ has a fundamental bounded resolution is easier.
As a simple  application of Theorem \ref{t:main} we prove the following

\begin{proposition}\label{cp}
For a Tychonoff space $X$ consider the following assertions:
\begin{enumerate}
\item $C_{p}(X)$ admits  a fundamental bounded resolution $\{B_{\alpha}: \alpha\in\omega^{\omega}\}$.
\item  $X$ is countable.
\item  $\mathbb{R}^{X}=\bigcup_{\alpha\in\omega^{\omega}}\overline{B_{\alpha}}^{\mathbb{R}^{X}}$ for a fundamental bounded resolution $\{B_{\alpha}: \alpha\in\omega^{\omega}\}$ in $C_{p}(X)$.
\item The strong (topological) dual $L_{\beta}(X)$  of $C_{p}(X)$  is a cosmic space, i.e. a continuous image of a metrizable separable space.
\item $C_{p}\left( X\right) $ is a large subspace of
$\mathbb{R}^{X}$, i.e. for every mapping $f \in \R^X$ there is a bounded set $B \subseteq C_p(X)$ such that
$f \in\overline{B}^{\mathbb{R}^X}$.
\end{enumerate}
Then $(1)\Leftrightarrow(2)\Leftrightarrow(3)\Leftrightarrow (4)\Rightarrow(5)$ but  $(5)\Rightarrow(2)$  fails even  for compact spaces $X$.
\end{proposition}

The implication $(1)\Ra(2)$ was recently  proved by Ferrando, Gabriyelyan and K\c akol  \cite{df} (with the help of $\cs^{*}$-networks). We will derive this implication from  Theorem \ref{t:main}.

\begin{proof}
(1) $\Rightarrow$ (2): If $C_{p}(X)$ has a fundamental bounded resolution $\{B_{\alpha}:\alpha\in\w^{\w}\}$, then the sets $U_{\alpha}=\{\xi\in L_{\beta}(X): \sup_{f\in B_\alpha}|\xi(f)|\leq 1\}$ form an $\omega^{\omega}$-base in
$L_{\beta}(X)$. By \cite{fe-ka-sa}, every bounded set in $L_\beta(X)$ is finite-dimensional. Applying Theorem \ref{t:main}, we conclude that the Hamel basis $X$ of the lcs $L_{\beta}(X)$ is countable. (2) $\Rightarrow$ (1) is clear.
(2) $\Rightarrow$ (3)$\wedge$(5): Since  $C_{p}(X)$ is dense in  the  metrizable  space $\mathbb{R}^X$, the claims hold. (2) $\Rightarrow$ (4): If $X$ is countable, then  $L_{\beta}(X)$ has a fundamental sequence of compact sets covering $L_{\beta}(X)$ and \cite[Proposition 7.7]{michael} implies that $L_{\beta}(X)$ is an $\aleph_{0}$-space, hence cosmic. (4)  $\Rightarrow$  (2): If  $L_{\beta}(X)$ is cosmic, then it is separable, and
\cite[Corollary 2.5]{fe-ka} shows that $X$ is countable. (5) $\nRightarrow$ (2):   $C_{p}(X)$ over every Eberlein scattered,  compact $X$ satisfies (5), see \cite{FKLS}.
\end{proof}

 Item (5) in Proposition~\ref{cp} is strictly connected with the following result.

\begin{theorem}{\rm (\cite{FKLS}, \cite{fe-ka})}\label{chara}
For a Tychonoff space $X$, the following conditions are equivalent{\rm:}
\begin{enumerate}
\item[{\rm (i)}] $C_p(X)$ is distinguished, i.e. the strong dual $L_{\beta}(X)$ of the space $C_{p}(X)$ is bornological.
\item [{\rm (ii)}] The strong dual $L_{\beta}(X)$ of the space $C_{p}(X)$ is a Montel space.
\item[{\rm (iii)}]  $C_{p}\left( X\right) $ is a large subspace of
$\mathbb{R}^{X}$.
\item[{\rm (iv)}] The strong dual $L_{\beta}(X)$ of the space $C_{p}(X)$ carries the finest locally convex topology.
\end{enumerate}
\end{theorem}


The following is a linear counterpart to item (4) in Proposition \ref{cp}.
\begin{remark}
A Tychonoff space $X$ is finite if and only if $L_{\beta}(X)$ is a continuous linear image of a metrizable lcs.
\end{remark}
Indeed, if  $X$ is finite, nothing  is left to prove. Conversely, assume that $L_{\beta}(X)$ is a continuous linear image of a metrizable lcs $E$ (by a one-to-one map). But $L_{\beta}(X)$ has only finite-dimensional bounded sets and $E$ fails this property. Hence $X$ is finite.

\end{document}